\documentclass[11pt]{amsart}

\usepackage{latexsym}
\usepackage{a4wide}
\usepackage{amscd}
\usepackage{graphics}
\usepackage[pdftex]{graphicx}
\usepackage{amsmath}
\usepackage{amssymb}
\usepackage{esint}
\usepackage{mathrsfs}
\usepackage{mathabx}	% for norm with triple bars
\usepackage{amsthm}
\usepackage{bbm}
\usepackage{color}
\usepackage{accents}
\usepackage{enumerate}
\usepackage{hyperref}
\input xy
\xyoption{all}

\newcommand{\N}{\mathbb{N}}                     % the natural numbers
\newcommand{\Z}{\mathbb{Z}}                     % the integer numbers
\newcommand{\R}{\mathbb{R}}                     % the real line
\newcommand{\D}{\mathbb{D}}                     % the disc
\newcommand{\varphitilde}{\tilde{\varphi}}		% tilde over varphi
\newcommand{\psitilde}{\tilde{\psi}}			% tilde over psi
\newcommand{\supp}{\mathrm{supp\,}}             % support
\newcommand{\spec}{\mathcal{S}}             % spectrum
\newcommand{\specm}{\mathcal{S}_{\mathrm{mean}}}            % mean spectrum
\newcommand{\specmbar}{\overline{\mathcal{S}}_{\mathrm{mean}}}            % mean spectrum closure
\newcommand{\Cal}{\mathrm{CAL\,}}   		% Calabi

\newcommand{\wind}{\mathrm{wind}}  	% winding number
\newcommand{\Ac}{\mathcal{A}}

\newcommand{\area}{\mathrm{Area}}
\newcommand{\Id}{\mathrm{Id}}

\theoremstyle{theorem}				
\theoremstyle{definition}					
\theoremstyle{remark}
\theoremstyle{theorem}
\newtheorem{thm}{\sc Theorem}%[section]		%numbered to restart with section
\newtheorem{lem}[thm]{\sc Lemma}			%numbered along with thm
\newtheorem{prop}[thm]{\sc Proposition}		%numbered along with thm
\newtheorem{cor}{\sc Corollary}			%numbered along with thm

\theoremstyle{definition}
\newtheorem{defn}[thm]{\sc Definition}		%numbered along with thm
\newtheorem{ex}{\sc Example}			%numbered along with thm
\theoremstyle{remark}
\newtheorem{rem}{Remark}		%numbered independently

\title[Spectral invariants for non-compactly supported Hamiltonians on the disc]{Spectral invariants for non-compactly supported Hamiltonians on the disc, and an application to the mean action spectrum}

\author{Barney Bramham, Abror Pirnapasov}

\address{Ruhr-Universit\"at Bochum}
\email{\text{barney.bramham@rub.de }}

\address{University of Maryland, College Park}
\email{\text{apirnapa@umd.edu}}

\date{\today}

\begin{document}

\maketitle

\begin{abstract}
For a symplectic isotopy on the two-dimensional disc we show that the classical spectral invariants of Viterbo \cite{viterbo1992symplectic} can be extended in a meaningful way to {\it non-compactly} supported
Hamiltonians.   We establish some basic properties of these extended invariants and as an application we show that Hutchings' inequality in \cite{hutchings2016mean}  between the Calabi invariant and the mean action spectrum 
holds without any assumptions on the isotopy; in \cite{hutchings2016mean} it is assumed that the Calabi invariant is less than the rotation number (or action) on the boundary.  

\end{abstract}

\section{Introduction}

We begin with the application.   In \cite{hutchings2016mean} Hutchings used embedded contact homology to prove a striking inequality relating the mean action spectrum with the Calabi invariant of 
a Hamiltonian isotopy of the disc.     Recall that if $\varphitilde=\{\varphi_t\}_{t\in[0,1]}$ is a smooth symplectic isotopy on a closed two-dimensional symplectic disc $(\D,\omega)$, 
starting at the identity $\varphi_0=\mathrm{id}_{\D}$, then the Calabi invariant $\Cal(\varphitilde)\in\R$ can be defined as an average action of $\varphitilde$, or equivalently as the integral 
$\int_0^1\int_{\D}H\,\omega\,dt$ of the unique generating Hamiltonian $H:\R/\Z\times\D\to\R$ that has been 
normalized to vanish pointwise on the boundary of the disc.   Unless stated otherwise, all objects are smooth.

\begin{thm}[Hutchings' inequality \cite{hutchings2016mean}]\label{T:Hutchings Intro}
Consider a symplectic isotopy $\varphitilde=\{\varphi_t\}_{t\in[0,1]}$ on the closed symplectic disc $(\D,\omega_0)$, normalized to have area $1$ \footnote{Some such normalization is necessary, as 
$\Cal(\varphitilde)$ has different dimensions to the action $\sigma_{\varphitilde,\lambda}$.  Alternatively, in the statement and throughout the paper, one could replace the Calabi invariant by the spatial average $\frac{1}{\int_{\D}\omega_0}\Cal(\varphitilde)$ and replace $a$ by the average $\frac{1}{\int_{\partial\D}\lambda_0}\int_{\partial\D}\sigma_{\varphitilde,\lambda_0}\lambda_0$. }.    Let $a:=\overline{\sigma}_{\varphitilde}|_{\partial\D}\in\R$ denote the mean action on the boundary.   If  
\begin{equation}\label{E:ass  Intro}
								\Cal(\varphitilde)< a
\end{equation}
then 
\begin{equation}\label{E:Hutchings  Intro}
								\inf\specm(\varphitilde)\ \leq\ \Cal(\varphitilde)
\end{equation}
where $\specm(\varphitilde)$ is the set of all mean actions at interior periodic points of $\varphi$, also called the mean spectral values.   See \ref{S:mean action} for precise definitions.  
\end{thm}
Note that $a$ is also the rotation number of $\varphitilde$ restricted to $\partial\D$, which is how the statement is formulated in \cite{hutchings2016mean}.   This agreement is due to the normalization of the action function chosen in \eqref{E:bc}, see also Remark \ref{R:bc}.   In \cite{hutchings2016mean} it is assumed that the time-$1$ map $\varphi$ is a rigid rotation on a neighborhood of the boundary of the disc.   This boundary 
condition was later removed by the second author, see \cite{pirnapasov2021hutchings}.   

Our main application in this article is that assumption \eqref{E:ass  Intro} can be removed: 

\begin{thm}\label{T:Hutchings strong 1}
For any symplectic isotopy $\varphitilde=\{\varphi_t\}_{t\in[0,1]}$ on the disc $(\D,\omega_0)$, normalized to have unit area, there holds 
\begin{equation*}
								\inf\specm(\varphitilde)\ \leq\ \Cal(\varphitilde).  
\end{equation*}
\end{thm}

This result can also be applied to the inverse isotopy $\varphitilde^{-1}=\{\varphi^{-1}_t\}_{t\in[0,1]}$, something which could not be done from the original statement, because condition \eqref{E:ass Intro} 
can never hold simultaneously for $\varphitilde$ and $\varphitilde^{-1}$.   Therefore:  

\begin{cor}\label{T:Hutchings strong 2}
For any symplectic isotopy $\varphitilde=\{\varphi_t\}_{t\in[0,1]}$ on the disc $(\D,\omega_0)$, normalized to have unit area, the associated Calabi invariant and the mean spectral values 
satisfy 
\begin{equation}\label{E:Hutchings strong 2}
								\inf\specm(\varphitilde)\ \leq\ \Cal(\varphitilde)\ \leq\ \sup\specm(\varphitilde).  
\end{equation}
\end{cor}
In \cite{Cal} Le Calvez proved inequality \eqref{E:Hutchings strong 2} for $C^{1}$ pseudo-rotations that are a rotation on the boundary using his method of transverse foliations.   In this case the inequalities are equalities. 

Let us point out that we do not give a new proof of Theorem \ref{T:Hutchings Intro}: 
The proof in \cite{hutchings2016mean} ultimately relies on deep connections 
between the ECH spectrum and contact volume in three dimensions, see \cite{cristofaro2015asymptotics}, and as far as we are aware no alternative proof is known; either using tools from dynamical systems or more classical Floer-theoretic methods.  An approach using periodic Floer homology (PFH) is also not known, 
despite connections between the PFH spectrum and the Calabi invariant \cite{cristofaro2021periodic}.   One difficulty is that the twisted PFH action is different to the Hamiltonian action.  However there are some related results:  In \cite{pirnapasov2022generic} the second author and Prasad used PFH to prove a weaker version of \eqref{E:Hutchings strong 2} for $C^{\infty}$-generic Hamiltonian 
diffeomorphisms of compact surfaces, possibly with boundary; in \cite{weiler2021mean, nelson2023torus} Weiler and Nelson-Weiler use ECH to obtain results similar to Theorem \ref{T:Hutchings Intro} for symplectic diffeomorphisms of the annulus and for higher-genus surfaces with one boundary component.

\subsection{Overview of extending the spectral invariants}

Recall that the classical spectral invariants \cite{viterbo1992symplectic} are a pair of continuous sections of the action spectrum $\varphitilde\mapsto c_-(\varphitilde)$ and $\varphitilde\mapsto c_+(\varphitilde)$, 
that are defined on any smooth Hamiltonian isotopy $\varphitilde=\{\varphi_t\}_{t\in[0,1]}$ on the standard symplectic vector space $(\R^{2n},\omega_0)$, provided the isotopy can be generated by a 
compactly supported, time-dependent Hamiltonian $H\in C^\infty(\R/\Z\times\R^{2n},\R)$.    This is recalled more precisely in \ref{S:classical spec}. 
 The {\it action spectrum} of $\varphitilde$ refers to the set of actions of the $1$-periodic orbits: 
\begin{equation}\label{E:action spec}
							\spec(\varphitilde):=\big\{ \Ac_{H}(p)\in\R\ \big|\ p\in\mathrm{Fix}(\varphi_1) \big\}
\end{equation}
where the action is defined by \eqref{D:action functional}.   The spectral invariants $c_-(\varphitilde)$ and $c_+(\varphitilde)$ enjoy many nice properties recalled in Theorem \ref{T:Vit 1}.  
In particular $c_-(\varphitilde)$ and $c_+(\varphitilde)$ depend only on the time-one map $\varphi_1$ and they depend continuously with respect to the Hofer metric.    

To extend $c_+$ and $c_-$ to a symplectic isotopy $\varphitilde$ on the disc, not necessarily compactly supported in the interior, we use a limiting procedure along the 
following lines: $\varphitilde$ is uniquely generated by a time-dependent Hamiltonian function on $\D$ that vanishes pointwise on the boundary $\partial\D$ for all times, even though in general it 
is not supported in the interior $\dot{\D}=\D\backslash\partial\D$.    Thus we can extend the Hamiltonian function continuously to $H:\R/\Z\times\R^2\to\R$ by setting $H\equiv0$ on the complement 
$\R^2\backslash\D$.    Thus $H$ is continuous everywhere, and smooth away from the circle $\partial\D\subset\R^2$.   
Then we take a sequence of smooth approximations $H_n:\R/\Z\times\R^2\to\R$ with compact support, which converge to $H$ (only) in a $C^0$-sense.    The spectral invariants for each $H_n$ 
are well defined and one can define 
\[
							c_{\pm}(\varphitilde):=\lim_{n\to\infty}c_{\pm}(H_n).   
\]  
 
Indeed, it follows easily from the continuity properties with respect to Hofer's norm, that the above limits exist and are independent of the choice of approximating sequence $H_n$.   The limits also trivially inherit 
a number of fundamental properties listed in Proposition \ref{P:1}, including the non-triviality property that at least one of $c_+(\varphitilde)$ and $c_-(\varphitilde)$ is non-zero if $\varphi\neq\Id$.   
However, two highly desirable properties are not obvious, because the convergence of the Hamiltonians is not in $C^1$:  
\begin{enumerate}
 \item First, that $c_{+}(\varphitilde)$ and $c_{-}(\varphitilde)$ depend only on the time-one map - or rather on the 
 homotopy class of $\varphitilde$ with fixed ends.   
 \item Secondly, that the values $c_{+}(\varphitilde)$, $c_{-}(\varphitilde)$
are indeed spectral values, i.e.\ that they represent the actions of fixed points of the time-one map.  
\end{enumerate} 
Note that for the classical spectral invariants, i.e.\ for compactly supported Hamiltonians on $\R^{2n}$, one usually shows the first property after establishing the second together with the 
continuity properties.    But this strategy doesn't work here as the second of the desired properties above will no longer be true in general.    What remains true, is that $c_{+}(\varphitilde)$, $c_{-}(\varphitilde)$ lie 
in a discrete set, which is enough.   We will show: 

\begin{thm}\label{T:spec main-1 Intro}
Let $\varphitilde=\{\varphi_t\}_{t\in[0,1]}$ be a symplectic isotopy on the disc $(\D,\omega)$ with area $A:=\mathrm{area}_\omega(\D)$ (not necessarily supported in $\D\backslash\partial\D$).   
Then the extended definition of spectral invariants $c_+(\varphitilde), c_-(\varphitilde)$ depend only on the homotopy class of $\varphitilde$ with fixed ends, and 
\begin{equation}
						c_+(\varphitilde)\in\spec(\varphitilde)\cup\{0\}\cup\{A\},
\end{equation}
and 
\begin{equation}
						c_-(\varphitilde)\in\spec(\varphitilde)\cup\{0\}\cup\{-A\}
\end{equation}
where $\spec(\varphitilde)\subset\R$ is the set of actions at the $1$-periodic orbits in $\D$.    
\end{thm}   

\begin{ex}\label{Ex:example not in spec}
There are simple examples on $(\D,\omega)$ where $c_{\pm}(\varphitilde)\notin \spec(\varphitilde)$, i.e.\ where the spectral invariants are not spectral values.    
Indeed, consider the case that $\varphitilde$ is the path of rotations on $\D$ about the origin through a total angle $\rho\in\R$.    
Then the single fixed point (at the origin) has action $A\rho$.    On the other hand, from standard properties we will see that 
$c_{+}(\varphitilde)\in [0,A]$.    Thus $c_{+}(\varphitilde)$ cannot be a spectral value if $\rho>1$ or if $\rho<0$.   Similarly $c_{-}(\varphitilde)$ cannot be a spectral 
value in this example if either $\rho>0$ or $\rho<-A$.   
\end{ex}

In contrast to Example \ref{Ex:example not in spec}, if the isotopy does not make more than a full rotation on the boundary of the disc, then 
at least one of the spectral invariants is a spectral value, i.e.\ represents the action of a fixed point, which is important for existence results.   More precisely:

\begin{thm}\label{T:spec main-2 Intro}
Let $\varphitilde=\{\varphi_t\}_{t\in[0,1]}$ be a symplectic isotopy on the disc $(\D,\omega)$, as in Theorem \ref{T:spec main-1 Intro}, 
and denote by $\rho:=\rho({\varphitilde}|_{\partial\D})$ the rotation number on the boundary.    Then: 
\begin{enumerate}[i)]  
  \item If $-1<\rho\leq1$ then $c_+(\varphitilde)\in\spec(\varphitilde)\cup\{0\}$.  
 \item If $0\leq\rho\leq1$ then $c_+(\varphitilde)\in\spec(\varphitilde)$.  
 \item If $-1\leq\rho<1$ then $c_-(\varphitilde)\in\spec(\varphitilde)\cup\{0\}$.
  \item If $-1\leq\rho\leq0$ then $c_-(\varphitilde)\in\spec(\varphitilde)$.  
\end{enumerate}
\end{thm}
By Example \ref{Ex:example not in spec} this statement is in some sense optimal.  

In proving Theorem \ref{T:spec main-2 Intro} we also use the following fact, which may be of independent interest:

\begin{prop}\label{P:rot}
For a symplectic isotopy on the disc $\varphitilde=\{\varphi_t\}_{t\in[0,1]}$ with rotation number $\rho:=\rho({\varphitilde}|_{\partial\D})$ on the boundary, we have: 
$\rho>0$ implies $c_+(\varphitilde)>0$, while $\rho<0$ implies $c_-(\varphitilde)<0$.
\end{prop}

\begin{rem}
We do not carry out any analogous program of extending the spectral invariants in higher dimensions.   This is a much more challenging problem 
and is not needed for our application.   
  
\end{rem}
\begin{rem} 
Note that the extension of the spectral invariants in this article is of a different nature to the {$C^0$}-extensions of the spectral invariants in 
\cite{oh2007locality,seyfaddini2013c,buhovsky2021action,kawamoto2022c}.   
Indeed, in these works 
the goal is to extend the spectral invariants to Hamiltonian homeomorphisms, i.e.\ the $C^0$-topology referred to is on the time-one maps, while in this article the 
$C^0$-convergence is on the level of the Hamiltonians.     
In particular, in general the time-one maps on the plane undergo a discontinuity in the limit, namely on the boundary of the disc.     
\end{rem}

\begin{rem}
In \cite{hutchings2016mean} Hutchings' inequality was stated for the mean action spectrum, rather than the {\it interior} mean action spectrum $\specm(\varphitilde)$.   This distinction is trivial in the statement of 
Theorem \ref{T:Hutchings Intro} because of the assumption \eqref{E:ass Intro}.   But it is relevant to Theorem \ref{T:Hutchings strong 1} and to Corollary \ref{T:Hutchings strong 2}.   It also 
has bearing on Remark 1.4 from \cite{hutchings2016mean}, where it is noted that the approach in \cite{hutchings2016mean} does not seem to prove the reverse inequality (i.e.\ the second inequality in \eqref{E:Hutchings strong 2}) 
and that this reverse inequality may be less interesting since, when the rotation number on the boundary is rational, the periodic points on the boundary 
will trivially satisfy the reverse inequality.    Hence we emphasize that Theorem \ref{T:Hutchings strong 1} and Corollary \ref{T:Hutchings strong 2} refer to {\it interior} periodic points, which make the statements somewhat more interesting.    
\end{rem}

\begin{rem}
The spectral invariants of Viterbo were extended by Schwarz and Oh \cite{schwarz2000action, oh2005construction} to Hamiltonian diffeomorphisms on 
closed symplectic manifolds using Hamiltonian Floer theory.    They were further extended by Frauenfelder-Schlenk \cite{frauenfelder2007hamiltonian} and Lanzat \cite{lanzat2016hamiltonian} to compact 
symplectic manifolds $(M,\omega)$ with contact type boundary (also called convex boundary).   In these latter cases the Hamiltonians are compactly supported in $M\backslash\partial M$,  in contrast, we allow non-compactly supported Hamiltonians, albeit restricted to this $2$-dimensional setting.   
\end{rem}

\subsection{Proof of Theorem \ref{T:Hutchings strong 1}}
Here we show how to remove the assumption on the Calabi invariant \eqref{E:ass  Intro}, by combining Theorem \ref{T:Hutchings Intro} of Hutchings' 
with the properties of the extended spectral invariants described above.   We begin with the following Corollary to Theorem \ref{T:spec main-2 Intro}: 

\begin{cor}\label{C:a in Smeanbar}
Consider a symplectic isotopy $\varphitilde=\{\varphi_t\}_{t\in[0,1]}$ as in Theorem \ref{T:spec main-1 Intro}.    Then the mean action on the boundary $a:=\overline{\sigma}_{\varphitilde}|_{\partial\D}\in\R$ lies in the closure of the interior mean action spectrum: 
\begin{equation}\label{E:a in Smeanbar}
				a\in\specmbar(\varphitilde). 
\end{equation} 
\end{cor}
\begin{proof} 
It suffices to show that $\varphi$ has an interior fixed point $x$ satisfying 
\begin{equation}\label{E:fixed point estimate}
				|\sigma_{\varphitilde}(x)-\pi k|\leq\pi
\end{equation} 
where $k\in\Z$ is unique so that $a\in[k\pi,(k+1)\pi)$.     Indeed, then applying \eqref{E:fixed point estimate} to each iterate $\varphitilde^j$ yields a sequence of $j$-periodic points $x_j\in\D\backslash\partial\D$ with action $|\sigma_{\varphitilde^j}(x)-\pi k_j|\leq\pi$, where $k_j\in\Z$ is such that $aj\in[k_j\pi,(k_j+1)\pi)$, which after dividing through by $j$ implies 
$\lim_{j\to\infty}\overline{\sigma}_{\varphitilde}(x_j)=\lim_{j\to\infty}\pi k_j/j=a$ as required.

To prove \eqref{E:fixed point estimate}, first suppose $a\in[0,\pi)$, i.e.\ $k=0$.   
If $\varphitilde$ is isotopic, with fixed ends, to the identity isotopy, then $a=0$ is also the action of each interior fixed point of $\varphitilde$, so that \eqref{E:a in Smeanbar} is trivially true.  
If $\varphitilde$ is not isotopic, with fixed ends, to the identity, then at least one of $c_+(\varphitilde)$ and $c_-(\varphitilde)$ is non-zero and so, due to our assumption on $a$, it follows from items (i) or (iii) in Theorem \ref{T:spec main-2 Intro} that this non-zero spectral invariant 
is represented by the action at a fixed point $x$.    This fixed point cannot be on the boundary; if it were, it would have action zero, whereas $x$ has non-zero action.   Therefore $|\sigma_{\varphitilde}(x)|=|\Ac_H(x)|$ equals $|c_+(\varphitilde)|\leq\pi$ or $|c_-(\varphitilde)|\leq\pi$.   Thus \eqref{E:fixed point estimate} holds, as $k=0$ in this case.  

Now suppose $a\notin[0,\pi)$, i.e.\ $k\neq0$.    Let $\tilde{R}_{2\pi k}$ denote the isotopy from the identity to the identity that rotates a total of $k$ times around the origin: 
$\tilde{R}_{2\pi k}=\{R_{2\pi kt}\}_{t\in[0,1]}$.  Then the mean action on the boundary of $\tilde{R}_{-2\pi k}\circ\varphitilde$ is $a-\pi k\in[0,\pi)$, so that we can apply 
the argument above to the isotopy $\tilde{R}_{-2\pi k}\circ\varphitilde$ and conclude that $\varphi$ has an interior fixed point $x$ satisfying $|\sigma_{\tilde{R}_{-2\pi k}\circ\varphitilde}(x)|\leq\pi$, 
in other words, $|\sigma_{\varphitilde}(x)-\pi k|\leq\pi$.  Thus \eqref{E:fixed point estimate} holds in general.  
\end{proof}

\begin{rem}
Note that inequality \eqref{E:fixed point estimate} can be interpreted as a quantitative version of Brouwer's fixed point theorem.   Curiously, applying 
it to the inverse isotopy can yield a slightly different inequality, depending on the boundary behavior.  

\end{rem}

\begin{proof}[Proof of Theorem \ref{T:Hutchings strong 1}]
Let $a:=\overline{\sigma}_{\varphitilde}|_{\partial\D}\in\R$ denote the mean action on the boundary.  By Hutchings' theorem, Theorem \ref{T:Hutchings Intro}, the conclusion 
holds if $\Cal(\varphitilde)< a$, so we may assume that $a\leq\Cal(\varphitilde)$.    By Corollary \ref{C:a in Smeanbar} $a$ lies in the closure of the mean action spectrum, 
thus 
\[
									\inf\specm(\varphitilde) \leq a\leq \Cal(\varphitilde)  
\]
as required.  
\end{proof}

Thus it remains to prove Theorems \ref{T:spec main-1 Intro} and \ref{T:spec main-2 Intro}, which is done in Section \ref{S:ext disc}, as is Proposition \ref{P:rot}.    
Section \ref{S:ext} discusses the limiting process used to extend the spectral invariants, and establishes some of their basic properties.  

\paragraph{Acknowledgments.} We would like to thank Alberto Abbondandolo for helpful discussions.    A.\ P. is greatful to Dan Cristofaro-Gardiner and Marco Mazzucchelli for their support. 
B.\ B. is partially supported by the DFG under the Collaborative Research Center SFB/TRR 191 - 281071066 (Symplectic Structures in Geometry, Algebra and Dynamics). A.\ P. is partially supported by the DFG Walter Benjamin Fellowship, Projektnummer 518128580.

\section{Preliminaries}

\subsection{The action and Calabi invariant for symplectic isotopies on the disc}\label{S:mean action}
In the following $\omega$ is a smooth symplectic form on the closed disc $\D\subset\R^2$, so % (non-degeneracy holds also at points on the boundary).   
without loss of generality $\omega=cdx\wedge dy$ for some constant $c>0$.    
We will denote by $\omega_0$ the normalization for which $\int_{\D}\omega_0=1$.    By a {\it symplectic isotopy} on $(\D,\omega)$ we mean a smooth path $\varphitilde=\{\varphi_t\}_{t\in[0,1]}$ of diffeomorphisms $\varphi_t:\D\to\D$, starting at the identity map $\varphi_0=\Id$, 
with $\varphi_t^*\omega=\omega$ for each $t\in[0,1]$.   Note that we are not assuming these maps are the identity on or near to the boundary of the disc.     
Given a smooth primitive $\lambda$ of $\omega$, the \textit{action function} of $\varphitilde$ 
\[
						\sigma=\sigma_{\varphitilde,\lambda}:\D\to\R
\]
is defined to be the unique solution to 
\begin{align}
			&d\sigma=\varphi^*\lambda - \lambda \qquad\ \ \ \ \mbox{on }\D\\
				&\sigma(z)=\int_{t\mapsto\varphi_t(z)}\lambda \qquad \mbox{for all }z\in\partial\D.  \label{E:bc}
\end{align}
It is not hard to show that $\sigma_{\varphitilde,\lambda}$ exists, is unique, and depends only on the isotopy class of $\varphitilde$ with fixed ends, \cite{abbondandolo2018sharp, mcduff2017introduction}.    
At fixed points $\sigma_{\varphitilde,\lambda}$ is independent of $\lambda$, so we may write $\sigma_{\varphitilde}(z)$.  
The {\it mean action} at a $k$-periodic point $z\in\D$ of $\varphi$ is defined by $\frac{1}{k}\sum_{j=0}^k\sigma_{\varphitilde,\lambda}(\varphi^j(z))$
%\[
%						\frac{1}{k}\sum_{j=0}^k\sigma_{\varphitilde,\lambda}(\varphi^j(z))
%\]
and is independent of $\lambda$, since it coincides with $(1/k)\sigma_{\varphitilde^k,\lambda}(z)$, being the value at a fixed point of $\varphi^k$.   
More generally, by the ergodic theorem,  the 
{\it asymptotic mean action function}
\[
	\overline{\sigma}_{\varphitilde}:\D\to\R,\qquad \overline{\sigma}_{\varphitilde}(z):=\lim_{k\to\infty}\frac{1}{k}\sum_{j=0}^k\sigma_{\varphitilde,\lambda}(\varphi^j(z))
\]
is well defined almost everywhere on $\D$ and at the periodic points.   It is also not hard to show that $\overline{\sigma}_{\varphitilde}$ is independent of $\lambda$, 
see Lemma 2.2 in \cite{abbondandolo2018sharp}, and that it is defined at each point on the boundary $\partial\D$, where it is constant and coincides with $\mathrm{area}(\D,\omega)\rho$, where $\rho\in\R$ is the rotation number 
of $\varphitilde$ restricted to the boundary of the disc.   

Finally, we define the {\it interior mean action spectrum} of $\varphitilde$ to be the set of mean actions at interior periodic points:  
\[
		\specm(\varphitilde):=\Big\{\ \overline{\sigma}_{\varphitilde}(z)\in\R\ \Big|\ z\in\D\backslash\partial\D,\ \ \varphi^k(z)=z\mbox{ some }k\in\N\ \Big\}. 
\]
Another important quantity is the space average of the action: 
\[
							\Cal_{\omega}(\varphitilde):=\int_{\D}\sigma_{\tilde{\varphi},\lambda}\,\omega, 
\]
called the Calabi invariant, which is infact independent of the primitive $\lambda$.   
\begin{rem}\label{R:bc}
The validity of Hutchings' inequality is independent of whether we add a constant to the action function $\sigma_{\varphitilde,\lambda}$.   In other words, the 
choice of the boundary condition \eqref{E:bc} is irrelevant.  
\end{rem}

\subsection{The classical spectral invariants on $(\R^{2n},\omega_0)$}\label{S:classical spec}
Let $\omega_0=\sum_{i=1}^ndx_i\wedge dy_i$ be the standard symplectic structure on $\R^{2n}$ and $\lambda_0$ the primitive $\frac{1}{2}\sum_{i=1}^n\big(x_i dy_i-y_idx_i\big)$.    
Consider a smooth and time-dependent Hamiltonian function $H:\R/\Z\times\R^{2n}\to\R$ with compact support.    

For any smooth loop $x\in C^\infty(\R/\Z,\R^{2n})$ the associated action is defined by 
\begin{equation}\label{D:action functional}
				\Ac_{H}(x) :=  \int_{\R/\Z}x^*\lambda_0 \ +\ \int_0^1 H(t,x(t))\,dt.  
\end{equation}
Formally the critical points of $\Ac_{H}$ correspond to the $1$-periodic orbits of the time-dependent Hamiltonian vector field defined by $i_{X_{H^t}}\omega=dH^t$, 
where $H^t(\cdot)=H(t,\cdot)$.    These sign conventions are convenient for two reasons; first, signs of rotation numbers of trajectories are aligned with the sign of the 
Hamiltonian - see for example the statement of Proposition \ref{P:rot}; secondly, so that the action function $\Ac_{H}$ on the loop space agrees with the action function 
from the previous section at periodic points in the following sense: Suppose $n=1$ and $H^t=0$ on $\partial\D$ for each $t\in[0,1]$.   Then $X_{H^t}$ is tangent to $\partial\D$, 
and so it generates a symplectic isotopy $\varphitilde=\{\varphi_t\}_{t\in[0,1]}$ on $\D$.    With these sign conventions, at each fixed point $p\in\D$ of the time-one map $\varphi$, 
we have 
\[
						\sigma_{\varphitilde,\lambda_0}(p)=\Ac_{H}(x)  
\]
where $x$ is the $1$-periodic orbit starting at $p$, i.e.\ $x(t)=\varphi_t(p)$.    For a proof see \cite{mcduff2017introduction} or \cite{abbondandolo2018sharp}.   
More generally, the set 
\[
				\spec(H):=\big\{ \Ac_{H}(x)\ \big|\ x\in \mathcal{P}_1(H) \big\}.
\]
where $\mathcal{P}_1(H)$ is the collection of $1$-periodic trajectories of $H$ in $\R^{2n}$, are called the {\it spectral values} of $H$.   
We also write $\spec(\varphitilde)$ instead of $\spec(H)$, see \eqref{E:action spec}, where $\varphitilde$ is the path of Hamiltonian diffeomorphisms generated by $H$.

Spectral invariants were first introduced by Viterbo \cite{viterbo1992symplectic}.   See also the so called spectral selector in Hofer-Zehnder \cite{hofer1990new}.  
The fundamental properties are contained in the following statement: 

\begin{thm}[\cite{viterbo1992symplectic}]\label{T:Vit 1}
For any compactly supported Hamiltonian diffeomorphism $\varphi=\varphi_H$ generated as the time-$1$ map of $H\in C_{c}^\infty(\R/\Z\times\R^{2n})$ the 
spectral invariants $c_+(\varphi)$ and $c_-(\varphi)$ have the following fundamental properties: 
\begin{enumerate}
 \item $c_+(\varphi^{-1})=-c_-(\varphi)$. 
 \item $c_-(\varphi)\leq 0\leq c_+(\varphi)$. 
 \item $c_+(\varphi\psi)\leq c_+(\varphi)+c_+(\psi)$. 
 \item $\gamma(\varphi)\geq\mathrm{area}(U)$ for every displaceable open disc $U$; meaning $\varphi(U)\cap U=\emptyset$. 
 \item $\gamma(\varphi)\leq\mathrm{area}(D)$ if $H$ has support in $D$. 
 \item $\gamma(\varphi)\leq \vvvert \varphi\vvvert$. 
 \item $H\mapsto c_+(\varphi_H)$ and hence $H\mapsto c_-(\varphi_H)$ and $H\mapsto \gamma(\varphi_H)$ are continuous, infact $1$-Lipschitz, with respect to the 
Hofer norm, see Definition \ref{D:Hofer norm}.   For example for $c_+$ this means  
\begin{equation}\label{E:Lipschitz}
			\vert c_+(\varphi_{H})-c_+(\varphi_{H'})\vert \leq \vvvert H - H'\vvvert
\end{equation}
for all $H, H'\in C_{c}^\infty([0,1]\times\R^{2n})$.
\end{enumerate}
When convenient we sometimes write $c_+(H)$ instead of $c_+(\varphi_{H})$, similarly for $c_-$.  
\end{thm}

\begin{defn}\label{D:Hofer norm}
Here $\gamma(\varphi):=c_+(\varphi)-c_-(\varphi)$ is called the $\gamma$-norm of $\varphi$, and 
\[
	\vvvert \varphi\vvvert:=\inf_{H}\vvvert H\vvvert
\]
is called the Hofer norm of $\varphi$, where the infimum is taken over all compactly supported $H$ with time-$1$ map $\varphi$ and 
$\vvvert H\vvvert:=\vvvert H\vvvert_{\R^{2n}}$ we call the Hofer norm of $H$, where more generally 
\[
	\vvvert H\vvvert_{E}:=\int_0^1\max_{z\in E}H(t,z)-\min_{z\in E}H(t,z)\, dt   
\]  
for any subset $E\subset\R^{2n}$.   
\end{defn}

Another important feature of the spectral invariants is the following well-known monotonicity property, which is typically used to prove Theorem \ref{T:Vit 1}.   

\begin{thm}[\cite{viterbo1992symplectic}]\label{T:Vit mon}
If $H, G\in C_{c}^\infty(\R/\Z\times\R^{2n})$ are two compactly supported Hamiltonians, generating diffeomorphisms $\varphi_H$ and $\varphi_G$ as the time-$1$ maps, 
then 
\begin{equation}\label{E:Vit mon}
					H\geq G \ \implies\ c_+(H)\geq c_+(G).  
\end{equation}
\end{thm}

\section{Extending $c_{-}(\varphi)$ and $c_+(\varphi)$ to non-compactly supported isotopies}\label{S:ext}
In this section we explain how to construct the spectral invariants for a non-compactly supported Hamiltonian on the disc and show that many of the usual properties 
for compactly supported Hamiltonians extend, see Proposition \ref{P:1} and Lemma \ref{L:mon}.    
Throughout this section $H\in C^\infty\big(\R/\Z\times\D\big)$ satisfies
\begin{equation}\label{E:bc2}
								H_t(z)=0\qquad \forall t\in\R/\Z,\ z\in\partial\D.
\end{equation}
Note that each symplectic isotopy $\varphitilde=\{\varphi_t\}_{t\in[0,1]}$ on the disc, starting at the identity, is generated by a unique such $H$.   
 
\begin{defn}\label{D:as}
We define an \emph{approximating sequence} for $H$ to be any sequence $H_n\in C_c^\infty\big(\R/\Z\times\R^2\big)$, for $n\in\N$, satisfying the following two properties: 
\begin{enumerate}
	\item (Extension) $H_n$ restricted to $[0,1]\times\D$ is equal to $H$.    
	\item (Convergence) $\lim_{n\to\infty} \|H_n\|_{C^0(\R^2\backslash\D)}=0$.   
\end{enumerate}
\end{defn}

It is straightforward to see that an approximating sequence always exists.   In Lemma \ref{L:ext} and Remark \ref{R:ext} below we give constructions with additional properties.   
Now set: 

\begin{defn}\label{D:extended spectral invariants} 
For any approximating sequence $(H_n)_{n\in\N}$ for $H$, define  
\[
	c_+(H):=\lim_{n\to\infty}c_+(H_n),\qquad  c_-(H):=\lim_{n\to\infty}c_-(H_n)
\]
where $c_+(H_n)$ and $c_-(H_n)$ are the standard spectral invariants for compactly supported Hamiltonians 
on $\R^{2}$, as referred to in Theorem \ref{T:Vit 1}.    
\end{defn}

This definition agrees with the original one if $H$ has compact support in $\D\backslash\partial\D$ (take $H_n=H$ to be the constant sequence), and is justified by: 

\begin{lem}
If $(H_n)_{n\in\N}$ is an approximating sequence for $H$ then $c_+(H_n)$ and $c_-(H_n)$ are both Cauchy sequences in $\R$.  
\end{lem}

Indeed, it follows that the limit $\lim_{n\to\infty}c_+(H_n)$ (respectively $\lim_{n\to\infty}c_-(H_n)$) is independent of the approximating sequence: If 
$H'_n$ and $H''_n$ are two approximating sequences for $H$ then so is the mixed sequence that has $H'_n$ for $n$ even and $H''_n$ for $n$ odd.   
The proof is trivial: 

\begin{proof} 
For each $j,k\in\N$, 
\begin{align*} 
				|c_+(H_j)-c_+(H_k)|\ &\leq\ \vvvert H_j - H_k\vvvert_{\R^2}
								=\ \vvvert H_j - H_k\vvvert_{\R^2\backslash\D}
								\leq\ 2\| H_j - H_k\|_{C^{0}(\R^2\backslash\D)}\to 0, 
\end{align*}
where the first inequality is item (7) in Theorem \ref{T:Vit 1} (since clearly $c_+(H)\leq\gamma(H)$).   The same reasoning applies to $c_-$.  
\end{proof}

\begin{prop}\label{P:1}
Suppose $\varphitilde=\{\varphi_t\}_{t\in[0,1]}$ is a smooth path of area preserving diffeomorphisms of the disc starting at the identity.    Then the 
associated spectral invariants $c_+(\varphitilde)$ and $c_-(\varphitilde)$ defined in Definition \ref{D:extended spectral invariants} satisfy: 
\begin{enumerate}
 \item $c_+(\varphitilde^{-1})=-c_-(\varphitilde)$. 
 \item $c_-(\varphitilde)\leq 0\leq c_+(\varphitilde)$. 
 \item $c_+(\varphitilde\psitilde)\leq c_+(\varphitilde)+c_+(\psitilde)$. 
 \item $\gamma(\varphitilde)\geq\mathrm{area}(U)$ for every displaceable open disc $U$; meaning $\varphi_1(U)\cap U=\emptyset$. 
 \item $\gamma(\varphitilde)\leq\pi$. 
 \item $\gamma(\varphitilde)\leq \vvvert \varphitilde\vvvert_{\D}$ (notation from Definition \ref{D:Hofer norm}). 
 \item $H\mapsto c_+(\varphitilde_H)$ and hence $H\mapsto c_-(\varphitilde_H)$ and $H\mapsto \gamma(\varphitilde_H)$ are $1$-Lipschitz with respect to the 
Hofer norm $\vvvert \cdot \vvvert_{\D}$, where the latter is in the sense of Definition \ref{D:Hofer norm} for Hamiltonians which may not be supported in the interior of the disc.  
\end{enumerate}
\end{prop}
\begin{proof}
The proofs of all properties except for Property 5 follow trivially from Theorem \ref{T:Vit 1} and Definition \ref{D:extended spectral invariants}.
Property 5 also follows easily from Property 5 of Theorem \ref{T:Vit 1} if we choose a sequence of approximating Hamiltonians whose support shrinks to that of the disc in the limit.   The existence of such an approximating sequence with shrinking supports is also easy, but we postpone this to the next Lemma, since we will 
there construct an approximating sequence with additional properties.    
\end{proof} 

For the proofs of Theorems \ref{T:spec main-1 Intro} and \ref{T:spec main-2 Intro} we will use an approximating sequence with additional properties: 

\begin{lem}\label{L:ext}
Suppose $H\in C^\infty\big(\R/\Z\times\D\big)$ vanishes identically on $\R/\Z\times\partial\D$.   Then there exists an approximating sequence $H_n$ for $H$ which satisfies 
\begin{enumerate}
	\item $\supp H_n\subset \D_{1+1/n}$, 
	\item $\sup_{n\in\N}\|\nabla H_n\|_{C^0(\R^2\backslash\D)}<\infty$. 
\end{enumerate}
\end{lem}
Note that if the supports converge to $\D$ then this is the best we can expect; in general the second order derivatives of the Hamiltonians $H_n$ will blow up.   
\begin{proof} 
Fix a smooth function $\rho:\R\to\R$ satisfying $\rho(r)=r$ on $r\leq0$, $\rho(r)\geq0$ on $r\geq0$, and $\rho(r)=0$ on $r\geq1$.    Then for each $n\in\N$ the 
function $\rho_n:\R\to\R$ given by 
\[
				\rho_n(r):=\frac{\rho(nr)}{n}
\]
satisfies: $\rho_n(r)=r$ on $r\leq0$, $\rho_n(r)=0$ on $r\geq1/n$, while on $r\geq0$ we have $\rho_n\geq0$, $\|{\rho_n}'\|_{C^0}\leq\|{\rho}'\|_{C^0}$ and $\|{\rho_n}\|_{C^0}\leq\|{\rho}\|_{C^0}/n$.    
Fix any smooth extension of $H$ to the whole plane, and denote this also by $H$.    In polar coordinates this is a map 
$H:\R/\Z\times[0,\infty)\times\R\to\R, (t,r,\theta)\mapsto H(t,r,\theta)$.   Now set 
\begin{equation}\label{D:H_n}
						H_n(t,1+r,\theta):=H(t,1+\rho_n(r),\theta)\qquad \forall n\in\N.  
\end{equation}
Then $H_n(t,\cdot)$ coincides with $H(t,\cdot)$ on the unit disc in $\R^2$ because $\rho_n(r)=r$ on $r\leq0$, and $H_n$  is everywhere smooth.   Moreover, 
$H_n$ has compact support, indeed, $H_n(t,\cdot)$ vanishes outside of $\D_{1+1/n}$ because $\rho_n(r)$ vanishes on $r\geq 1/n$ and 
$H(t,\cdot)$ vanishes $\partial\D$.    The $\theta$-derivative is 
independent of $n$ so it suffices to bound the $r$-derivative:  From the chain rule 
$\|\partial_rH_n\|_{C^0}\leq \|\nabla H\|_{C^0}\|{\rho_n}'\|_{C^0}\leq \|\nabla H\|\|{\rho}'\|_{C^0}$ 
is uniformly bounded in $n$.   Finally $\|H_n(t,\cdot)\|_{C^0(\R^2\backslash\dot{\D})}\to 0$ for $n\to\infty$, uniformly in $t$, because 
\begin{align*}
		\sup\big\{ |H_n(t,1+r,\theta)|\ \big|\ 0\leq r \big\}\ &=\ \sup\big\{ |H(t,1+\rho_n(r),\theta)|\ \big|\ 0\leq r\big\} \\
										&=\ \sup\big\{ |H(t,1+r,\theta)|\ \big|\ 0\leq r\leq \|{\rho_n}\|_{C^0}\big\}
\end{align*}
which decays to zero as $n\to\infty$, because $\|{\rho_n}\|_{C^0}\leq\|{\rho}\|_{C^0}/n\to 0$ and $H$ is continuous and vanishes on $[0,1]\times\partial\D$.  
\end{proof}

\begin{rem}\label{R:ext}
In the proof of Lemma \ref{L:wind} it will be convenient to use a more refined sequence of approximations $H_n$.   

 In the refined version $H_n$ still has the form of \eqref{D:H_n} and satisfies the conditions of Lemma \ref{L:ext}, but now each $\rho_n$ satisfies 
\begin{equation}\label{E:extra condition rho}
											\rho'_n(r)\geq -1/n
\end{equation}
in addition to the previous properties: $\rho_n(r)=r$ on $r\leq0$, $\rho_n(r)=0$ on $r\geq1/n$, while on $r\geq0$ we have $\rho_n\geq0$, $\|{\rho_n}'\|_{C^0}\leq 1$ and $\|{\rho_n}\|_{C^0}\leq 1/n$.   
 
One way to achieve this with \eqref{E:extra condition rho} is to set $\rho_n(r):=\tilde{\rho}_n(nr)/n$, where $\tilde{\rho}_n:\R\to\R$ is a smooth function satisfying 
the following properties: $\tilde{\rho}_n(r)=r$ on $r\leq 0$, $\tilde{\rho}_n(r)\geq0$ on $r\geq0$ and $\tilde{\rho}_n(r)=0$ on $r\geq1$, and $1\geq\tilde{\rho}_n'(r)\geq-1/n$ for all $r$, which can easily 
be achieved.  
% by making $\tilde{\rho}_n(r)$ have a maximum of order $1/n$ soon after $r=0$ and then it can descend slower and still vanish identically by $r=1$.   
\end{rem}

We end this section by showing that the monotonicity property of Theorem \ref{T:Vit mon} extends to non-compactly supported Hamiltonians:  

\begin{lem}\label{L:mon}
Suppose $H, G\in C_{c}^\infty(\R/\Z\times\D)$ are two Hamiltonians on the disc, both satisfying the boundary condition \eqref{E:bc2}.   Then the extended spectral invariants satisfy 
\begin{equation}\label{E:mon}
					H\geq G \ \implies\ c_+(H)\geq c_+(G).  
\end{equation}
\end{lem}
In general the condition $H\geq G$ on $\D$ cannot hold for smooth extensions to the plane.   So there is something small to check.  
\begin{proof} 
Assume $H\geq G$ and let $H_n$ and $G_n$ be approximating sequences with supports in $\D_{1+1/n}$ as in Lemma  \ref{L:ext}.   
Let $\chi:\R^2\to[0,1]$ be a smooth function with compact support with $\chi\equiv1$ on the disc $\D_2$ of radius $2$.  For $\epsilon>0$ sufficiently small $2\epsilon\chi$ is 
sufficiently $C^2$-small that $c_+(2\epsilon\chi)\leq\max2\epsilon\chi =2\epsilon$.    

For $n_0\in\N$ sufficiently large we have $\|G_n|_{\R^2\backslash\D}\|_{\infty}\leq\epsilon$ and $\|H_n|_{\R^2\backslash\D}\|_{\infty}\leq\epsilon$ for all $n\geq n_0$, and therefore that $G_n$ and $H_n$ are smooth extensions of $G$ and $H$ respectively to the plane, having compact support and satisfying $G_n\leq H_n+2\epsilon\chi$ on $\R^2$.   So by Theorem \ref{T:Vit mon} $c_+(G_n)\leq c_+(H_n+2\epsilon\chi)$.  
Therefore, for all $n\geq n_0$, 
\[
		c_+(G_n)\leq c_+(H_n+2\epsilon\chi)\leq c_+(H_n)+c_+(2\epsilon\chi)\leq c_+(H_n)+2\epsilon
\]
where the second inequality is from the Lipschitz continuity of $c_+$ with respect to the Hofer norm, \eqref{E:Lipschitz}.   Letting $n\to\infty$ we 
conclude $c_+(G)\leq c_+(H)+2\epsilon$ for each $\epsilon>0$ as required.    

\end{proof} 

\section{Proof of Theorems \ref{T:spec main-1 Intro} and \ref{T:spec main-2 Intro}}\label{S:ext disc}

Now we can prove Theorem \ref{T:spec main-1 Intro} using an approximating sequence of Hamiltonians that have the properties provided by Lemma \ref{L:ext}.   Throughout this section we 
assume for convenience that the area of the disc is $\pi$, which leads to Theorem \ref{T:spec main-1 Intro} with $A=\pi$.   
We begin with:  

\begin{lem}\label{L:degree length}
For any $C^1$-differentiable loop $x:[0,1]\to\D_{1+\delta}\backslash\dot{\D}$, $\delta>0$, there holds 
\begin{equation}\label{E:degree length}
					\left| \int_{D}u^*\omega - \pi\deg(x)\right|\ \leq\ L(x)\delta
\end{equation}
where $u$ is any capping disc for $x$, and $\deg(x)\in\Z$ is the winding number of $x$ about the disc, i.e.\ about any point in $\D$, and $L(x)$ is the Euclidean length of $x$.  
\end{lem}
\begin{proof}
Let $x:[0,1]\to\D_{1+\delta}\backslash\dot{\D}$ be a continuously differentiable closed curve.   
 Note that each of the three terms in \eqref{E:degree length} is additive with respect to concatenation of closed curves, and that $x$ can be written as a finite sum (more precisely a concatenation) of simple closed curves, because it is $C^1$.    Hence it suffices to prove \eqref{E:degree length} for {\it simple} closed curves.
 Therefore let us suppose that $x$ is simple, i.e.\ $x(t_1)=x(t_2)$ implies $t_1=t_2$ or $t_1,t_2\in\{0,1\}$.    
 If $x$ is contractible in $\D_{1+\delta}\backslash\dot{\D}$, then we can lift it to a simple closed curve $\tilde{x}$ in the covering $\R\times[0,\delta]$ 
 via the projection $\pi(u,v)=(1+v)e^{2\pi iu}$.   
 If the image of $\tilde{x}$ lies in the rectangle $[a,b]\times[0,\delta]$, then it encloses an area at most $2\pi(b-a)\delta$ with respect to $\pi^*\omega$.   
 Clearly $2\pi(b-a)$ is bounded by half of the (euclidean) length of $\tilde{x}$ in $\R\times[0,\delta]$, which is bounded by $L(x)/2$.  
So the area enclosed by $x$ equals the $\pi^*\omega$-area enclosed by $\tilde{x}$, is bounded by $L(x)\delta/2$, i.e., 
  $\left| \int_{D}u^*\omega - \pi\deg(x)\right|\leq L(x)\delta/2$ as the degree term vanishes.    
 If however the simple closed curve $x$ is non-contractible in $\D_{1+\delta}\backslash\dot{\D}$ then $\deg(x)\in\{\pm1\}$ and we see that $\left| \int_{D}u^*\omega - \pi\deg(x)\right|\leq (1+\delta)^2\pi - \pi \leq (2\pi)\delta\leq L(x)\delta$.   This completes the proof of \eqref{E:degree length} when $x$ is simple, and therefore also the general case.   
\end{proof}

\begin{proof}(Of Theorem \ref{T:spec main-1 Intro})
It suffices to prove the statements for $c_+$, as the statement for $c_-$ follows then from the first property of Proposition \ref{P:1}.   For $c_+$ it suffices to show that 
\begin{equation}\label{spec}
		c_+(\varphitilde)\in\spec(\varphitilde)\cup\{0\}\cup\{\pi\}
\end{equation}
and the usual arguments will then imply that $c_+(\varphitilde)$ is invariant under 
homotopies of $\varphitilde$ through isotopies with fixed endpoints:   Indeed, $\spec(\varphitilde)\cup\{0\}\cup\{\pi\}$ is a closed nowhere dense set (closed is 
straightforward to see, and after extending a generating 
Hamiltonian for $\varphitilde$ to all of $\R^2$ with compact support it is a classical result that the spectrum is nowhere dense, see for example \cite{schwarz2000action}, and the extension only 
increases the set of spectral values) and we already showed that $c_+(\varphitilde)$ depends continuously on $\varphitilde$ in the $C^0$-topology.

We now show \eqref{spec}.    Fix a generating Hamiltonian $H\in C^\infty\big(\R/\Z\times\D\big)$ with $H_t(z)=0$ for all $t\in[0,1]$,  $z\in\partial\D$, 
for the given path $\varphitilde=\{\varphi_t\}_{t\in[0,1]}$.    Let $H_n\in C^\infty\big(\R/\Z\times\R^2\big)$ for $n\in\N$ be an approximating sequence for $H$ as in 
Lemma \ref{L:ext}.    In particular $\supp H_n\subset \D_{1+1/n}$ and the gradients are uniformly bounded.   
Since $\lim_{n\to\infty}c_+(H_n)=c_+(H)$ and $0\leq c_+(H_n)\leq\area(\supp H_n)\to\pi$ it follows that $c_+(H)\in[0,\pi]$.     
Therefore let us assume that $0<c_+(H)<\pi$ and show that $c_+(H)\in\spec(\varphitilde)$.   
For sufficiently large $n$ we must have 
\begin{equation}\label{E:bds Hn}
								\varepsilon<c_+(H_n)<\pi-\varepsilon
\end{equation}
for some sufficiently small $\varepsilon>0$.    By definition we find for each $n\in\N$ a $1$-periodic orbit $x_n\in C^{\infty}(\R/\Z,\R^2)$ for $H_n$ with action 
$\Ac_{H}(x_n)=c_+(H_n)$.   By \eqref{E:bds Hn} this is non-zero, so $x_n$ lies in the support of $H_n$, namely in $\D_{1+1/n}$.    Now if $x_n$ is contained in $\D$ for some subsequence, then we are done, since $H_n$ agrees with $H$ on $\D$.    So without loss of generality let us assume 
$x_n\subset A_n:=\D_{1+1/n}\backslash\D$ for all $ n\in\N$.  
Using \eqref{D:action functional}   
\begin{align}\label{E:c_+}
			c_+(H_n) &= \int_0^1 H_n(t,x_n(t))\,dt \ +\ \int_{D}u_n^*\omega
\end{align}
where $u_n:D\to\R^2$ is any capping disc for $x_n$.   As $n\to\infty$ the first term on the right decays to zero, because $H_n$ converges uniformly to zero outside of the 
interior $\dot{\D}$.  
It remains to understand the second term on the right of \eqref{E:c_+}.   
Applying Lemma \ref{L:degree length} we have 
\begin{equation*}
					\left| \int_{D}u_n^*\omega - \pi\deg(x_n)\right|\ \leq\ L(x_n)\frac{1}{n}
\end{equation*}
for each $n\in\N$.   As $\|\dot{x}_n\|=\|X_{H_n}(x_n)\|$ are uniformly bounded, so is $L(x_n)$ uniformly bounded.   
Thus letting $n\to\infty$ in the last line yields $\lim_{n\to\infty}d(\int_{D}u_n^*\omega,\pi\Z)=0$, where $d(\cdot,\cdot)$ denotes the distance, and so from \eqref{E:c_+} we have $\lim_{n\to\infty}d(c_+(H_n),\pi\Z)=0$, contradicting  \eqref{E:bds Hn}.   This completes the proof.
\end{proof}

Now we prove Proposition \ref{P:rot}: 

\begin{proof}(Of Proposition \ref{P:rot})
Suppose $\rho>0$.   By Theorem \ref{T:spec main-1 Intro} $c_+(\varphitilde)$ depends only on the homotopy class of $\varphitilde$ with fixed ends.   So we may assume, 
after such a homotopy, that the Hamiltonian vector field is autonomous on the boundary of the disc - this is stronger than necessary, it suffices that $X_{H^t}$ is nowhere vanishing on $\partial\D$.   
It follows that at each point $p\in\partial\D$, we have 
$-\partial_rH^t(p)=d\theta[X_{H^t}]>0$ on $\partial\D$,  and therefore there exists $\epsilon>0$ so that 
\begin{equation}\label{E:H>0 near b}
										H^t(z)>0\qquad \mbox{for all}\quad 1-2\epsilon\leq |z|^2<1.
\end{equation}
We now construct a Hamiltonian function $G$ on the disc, so that $G\leq H$ and so that $c_+(\varphi_{G})>0$.   Then the result will follow from 
Lemma \ref{L:mon}.   

We construct $G$ to be time independent, compactly supported in $\dot{\D}$, and radial.    
Indeed, for each parameter $\lambda\leq 0$ consider $G_{\lambda}(z)=g_{\lambda}(|z|^2)$ for some smooth function $g_{\lambda}:[0,\infty)\to\R$ whose 
graph has the following features: $g_{0}\geq 0$ is compactly supported in the interior of the region $1-\epsilon\leq r\leq 1$ and has a single positive bump, meaning it has 
a unique maximum and no other critical points in its support.  We can arrange that $g_0$ is sufficiently small that $G_0<H$ on the annulus $1-2\epsilon\leq|z|^2<1$, which 
is possible because of \eqref{E:H>0 near b}.    For $\lambda\leq0$ we choose $g_{\lambda}$ to satisfy the following: $g_{\lambda}=\lambda$ on $0\leq r\leq 1-2\epsilon$, 
$g_{\lambda}$ is a monotonic increasing function on $1-2\epsilon\leq r\leq 1-\epsilon$ with $g_{\lambda}=0$ near to $r=1-\epsilon$, and $g_{\lambda}=g_0$ on $r\geq 1-\epsilon$.  
For example, fix $g_0$ and $g_{-1}$ satifying the above properties, then set $g_{\lambda}=|\lambda|g_{-1}$ on $r\leq 1-\varepsilon$ and $g_{\lambda}=g_0$ on $r\geq 1-\epsilon$, then 
the dependence on $\lambda$ is also smooth.   
\begin{figure}[htbp]
\includegraphics[scale=0.8]{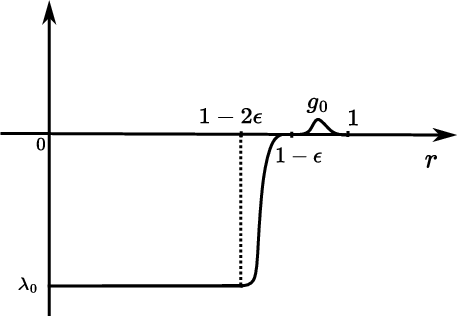}
\centering
\caption{Here the graph of $g_{\lambda}$ at $\lambda=\lambda_{0}$}
\label{fig:uni}
\end{figure}
Clearly the maximum value $M>0$ is independent of $\lambda\leq0$.     For $\lambda=0$ we have $G_0\geq0$, and it is somewhere strictly positive so $c_+(G_0)>0$.   
Therefore $c_+(G_0)$ is a positive spectral value, i.e.\ lies in $\spec(G_0)\cap(0,\infty)$.   More generally for $\lambda\leq0$, we know that $G_{\lambda}$ lies in $\spec(G_{\lambda})\cap[0,\infty)$.  
Since $G_\lambda$ is radial it follows from a computation that the values in $\spec(G_{\lambda})\cap[0,\infty)$ are precisely the points on the non-negative $y$-axis where each tangent line to the graph 
of $g_{\lambda}$ having integer slope $k\geq0$, meets the $y$-axis, see \cite{humiliere2016towards}. From this it is easy to see that $\spec(G_{\lambda})\cap[0,\infty)$ 
lies in the set $\{0\}\cup [M,\infty)$.   Therefore $\lambda\mapsto c_+(G_\lambda)$ is a continuous map into $\{0\}\cup [M,\infty)$ which takes the value $M>0$ at $\lambda=0$.   
Hence $c_+(G_\lambda)\geq M>0$ for all $\lambda\leq 0$.    Finally we set $G=G_{\lambda}$ for some $\lambda<\min\{ H^t(z)\, |\, t\in\R/\Z\}$.     Then $G\leq H$.   

If $\rho<0$ then the result follows from the case above together with item 1 of Proposition \ref{P:1}.  
\end{proof}

To prove Theorem \ref{T:spec main-2 Intro} we will use: 
\begin{lem}\label{L:wind}

Suppose that the rotation number on the boundary satisfies $|\rho(\varphitilde|_{\partial\D})|<1$, 
and that the generating Hamiltonian $H$ has autonomous Hamiltonian vector field on the boundary of the disc.   Let $H_n$ be the approximating sequence of Lemma \ref{L:ext} with the extra property described in Remark \ref{R:ext}.   Then if $n$ is sufficiently large, every time-$1$ orbit segment $x_n:[0,1]\to\R^2$ for $H_n$ that lies in $\R^2\backslash\dot{\D}$ has winding about the origin satisfying $|\wind(x_n)|<1$. 
\end{lem} 
In the above statement $\wind(x_n):=(\theta(1)-\theta(0))/2\pi$, if $x_n(t)=(r(t),e^{i\theta(t)})$. 

\begin{proof}
By assumption each $X_{H}$-trajectory on the boundary of the disc $t\mapsto(1,\theta):[0,1]\to\partial\D$ satisfies $|\theta(1)-\theta(0)|/2\pi<1-\epsilon$ for some 
uniform $\epsilon>0$.    We consider two cases separately.  \\

{\bf Case 1:} Here we consider the case that $\partial_rH(1,\theta)\geq 0$ for all $\theta\in\R$.    The case $\partial_rH(1,\theta)\leq 0$ can be handled similarly.  

Note that the vector field $X_{H^t}$ at a point $(1,\theta)\in\partial\D$ has angular component 
$d\theta\big[X_{H^t}(1,\theta)\big]=-\partial_rH^t(1,\theta)$.  
Applying Gronwall's Lemma to the vector fields $\partial_rH|_{\partial\D}$ and $-\partial_rH|_{\partial\D}$ there exists $\delta>0$ so that if $V$ is any smooth vector field on $\partial\D$ with $\|V-\partial_rH|_{\partial\D}\|<2\delta$ or $\|V+\partial_rH|_{\partial\D}\|<2\delta$, then 
 each $V$-trajectory $[0,1]\to\partial\D$, $t\mapsto(1,\psi(t))$ satisfies $|\psi(1)-\psi(0)|/2\pi<1-\epsilon/2<1$.    
 
Recall, as in Lemma \ref{L:ext}, that by slight abuse of notation, $H$ also refers to a fixed smooth extension (of $H$) to all of $\R^2$.    
Choose $n_0\in\N$ sufficiently large that 
\begin{equation}\label{E:n0 condition}
|\partial_rH^t(1+\rho_n(r),\theta)-\partial_rH(1,\theta)|<\delta
\end{equation}
for all $r\geq0$, $\theta\in\R$, $t\in[0,1]$, $n\geq n_0$.   The vector field $X_{H^t_n}$ has angular component 
$d\theta\big[X_{H^t_n}(1+r,\theta)\big]=-\partial_rH^t_n(1+r,\theta)=-\rho'_n(r)\partial_rH^t(1+\rho_n(r),\theta)$, and so  
\begin{equation}\label{E:bound on angular component}
	\left|d\theta\big[X_{H^t_n}(1+r,\theta)\big]\right|\leq  \left|\partial_rH(1,\theta)\right|+\delta
\end{equation}
by \eqref{E:n0 condition}, when $n\geq n_0$.   
Consider an arbitrary $X_{H^t_n}$-trajectory $x(t)=(r(t),\theta(t))$ in $\R^2\backslash\dot{\D}$.    Then by \eqref{E:bound on angular component} and 
our assumption that $\partial_rH(1,\theta)\geq0$ everywhere, 
\[
					-\partial_rH(1,\theta(t))-\delta\leq \dot{\theta}(t)\leq \partial_rH(1,\theta(t))+\delta
\] 
for all $t$.   Hence $t\mapsto\theta(t)$ is squeezed: $\psi_-(t)\leq \theta(t)\leq\psi_+(t)$, 
 where $\dot{\psi}_+(t)=\partial_rH(1,\psi_+(t))+\delta$ with initial condition $\psi_+(0)=\theta(0)$, and $\dot{\psi}_-(t)=-\partial_rH(1,\psi_-(t))-\delta$ with 
 $\psi_-(0)=\theta(0)$.    Thus $\wind(x)\leq (\psi_+(1)-\psi_+(0))/2\pi<1$, where the last inequality is from Gronwall, since the vector field 
$V(1,\theta)=\partial_rH(1,\theta)+\delta$ is within $\delta$ of $\partial_rH(1,\theta)$.   
  Similarly  $\wind(x)\geq (\psi_-(1)-\psi_-(0))/2\pi>-1$ where the last inequality is from Gronwall, since $V(1,\theta)=-\partial_rH(1,\theta)-\delta$ is within 
  $\delta$ of $-\partial_rH(1,\theta)$.     This proves the Lemma in Case 1.  \\
 
 {\bf Case 2:} Here we consider the case that $\partial_rH(1,\theta)$ changes sign: Then the sets 
 \[
 	\Omega^+=\{ (1,\theta)\in\partial\D | \partial_rH(1,\theta)>0 \},\qquad\Omega^-=\{ (1,\theta)\in\partial\D | \partial_rH(1,\theta)<0 \}
\]
are both non-empty and open.  We claim that for any $\delta>0$ there exists $n_0\in\N$ sufficiently large so that 
if $x(t)=(r(t),\theta(r))$ is an $X_{H^t_n}$-trajectory in $\R^2\backslash \dot{\D}$, and $n\geq n_0$, then 
\begin{equation}\label{E:speed in Omegas}
		\qquad \theta(t)\in\Omega^-\implies \dot{\theta}(t)\geq-\delta,\qquad\theta(t)\in\Omega^+\implies \dot{\theta}(t)\leq\delta.
\end{equation}
Assuming this claim for the moment, let's see how to complete the proof of the Lemma in Case 2:  
Fix two intervals of positive length $I^+\subset\Omega^+$, $I^-\subset\Omega^-$, and choose $\delta>0$ so that $\delta<(1/2)\min\{|I^+|, |I^-|\}$.    If $n$ is sufficiently large then by 
the first statement in \eqref{E:speed in Omegas} $t\mapsto\theta(t)$ can never get more than half way across $I^-$ in a negative direction, so $\wind(x_n)\geq -1+(1/2)|I^-|$.   
Similarly, from the second statement in \eqref{E:speed in Omegas}, $t\mapsto\theta(t)$ can never get more than half way 
across $I^+$ in a positive direction, so $\wind(x_n)\leq 1-(1/2)|I^+|$.   Thus $|\wind(x_n)|<1$.   

It remains to prove the claim \eqref{E:speed in Omegas}.   Let us show $\theta(t)\in\Omega^-\implies \dot{\theta}(t)\geq-\delta$, as the proof of the second inequality in 
\eqref{E:speed in Omegas} is analogous.  
Since $\partial_rH^t|_{\partial\D}=\partial_rH|_{\partial\D}<0$ on $\Omega^-$, there exists $\epsilon_0>0$ so that 
$-M\leq \partial_rH^t(1+r,\theta)\leq\delta$ for all $0\leq r\leq\epsilon_0$ and all $\theta\in\Omega^-$, where $M:=2\|\partial_rH|_{\partial\D}\|_{\infty}$.   Choose $n_0\in\N$ sufficiently large that $\max\rho_{n_0}<\epsilon_0$ and $\min\rho'_{n_0}>-\delta/M$ (this latter condition is possible because we are using 
$\rho_n$ as in Remark \ref{R:ext}).    Now, suppose $\theta(t)\in\Omega^-$ and $n\geq n_0$.   Then 
\begin{equation}\label{E:case 2 speed theta}
		\dot{\theta}(t)=d\theta\big[X_{H^t_n}(1+r(t),\theta(t))\big]=\rho_n'(r(t))\big[-\partial_rH^t(1+\rho_n(r(t)),\theta(t))\big].  
\end{equation}
If $0\leq\rho'_n(r(t))\leq1$ then the right hand side is bounded from below by $-\partial_rH^t(1+\rho_n(r(t)),\theta(t))$, which in turn is bounded from 
below by $-\delta$ because $0\leq\rho_n(r(t))\leq\max\rho_{n_0}<\epsilon_0$.    The alternative is that $0>\rho'_n(r(t))\geq\min\rho'_{n}>-\delta/M$. 
In this case the right hand side of \eqref{E:case 2 speed theta} is bounded from below by $(-\delta/M)M=-\delta$, as required. 
\end{proof}

\begin{proof}(Of Theorem \ref{T:spec main-2 Intro})
%Suppose that $c_+(\varphitilde)\neq0$, 

We will prove statements (i) and (ii) which concern $c_+(\varphitilde)$, as $c_-$ then follows from the first item in Proposition \ref{P:1}.     

By Theorem \ref{T:spec main-1 Intro} $c_+(\varphitilde)$ depends only on the homotopy class of $\varphitilde$ with fixed ends.   So we may assume, after such a homotopy, that the Hamiltonian vector field is autonomous on the boundary of the disc.   We do this to apply Lemma \ref{L:wind}.   

First we show (i) when $\rho=1$.   In this case the boundary of the disc has a fixed point and it will have action $A$, so $A$ lies in the spectrum.   So then (i) follows from 
Theorem \ref{T:spec main-1 Intro}.   

Now we show (i) when $|\rho|<1$.   
Fix an approximating sequence of Hamiltonians $H_n$ for $H$, that satisfy the conditions of Lemma \ref{L:ext} as well as those in Remark \ref{R:ext}.   
We have $c_+(H_n)=\Ac_{H_n}(x_n)\to c_+(\varphitilde)$ for $n\to\infty$, 
for some $1$-periodic orbits $x_n$ of $H_n$.    If each $x_n$ lies in the disc $\D$, then $c_+(H_n)=\Ac_{H_n}(x_n)=\Ac_{H}(x_n)$ is a spectral value of $H$, so we are done, since $\spec(H)$ is closed\footnote{Or we 
note that the sequence $H_n$ lies in a continuous family of Hamiltonians with $n\in[1,\infty)$ (since $\rho_n$ makes sense for non-integer values of $n$), and so by continuity of the spectral invariants and discreteness of the spectrum of $H|_{\D}$ we conclude that the sequence $c_+(H_n)$ is constant in $n$ since each $x_n$ lies in $\D$.}.    
The alternative, after restricting to a subsequence,  is that each $x_n$ is in $\R^2\backslash\D$.    
 
Since $|\rho|<1$ we have $|\wind(x_n)|<1$ by Lemma \ref{L:wind}, i.e.\ $\deg(x_n)=0$, for each $n$.    It follows from Lemma \ref{L:degree length} that $|\int_{D}u_n^*\omega|\to 0$ (the length of each $x_n$ is uniformly bounded because 
 $\nabla H_n$ is uniformly bounded).   Clearly $\int_0^1H_n^t(x_n(t))dt\to 0$ also, because $H_n\to 0$ uniformly on $\R^2\backslash\dot{\D}$.   
 We conclude that $c_+(\tilde{\varphi})=\lim_{n\to\infty}\Ac_{H_n}(x_n)= 0$ and again we are done.    This completes the proof of statement (i).  
 
 To show statement (ii):  If $\rho=0$ then due to a fixed point on the boundary we know that zero lies 
 in the spectrum, so the statement follows from (i).   If instead $\rho>0$, then by Proposition \ref{P:rot} 
 again the statement follows from (i).   
\end{proof}

\bibliographystyle{plain}
\bibliography{reference}

\end{document}